\documentclass[oneside,english]{amsart}
\usepackage[T1]{fontenc}
\usepackage[latin9]{inputenc}
\usepackage{amsthm}
\usepackage{amssymb}
\usepackage{setspace}
\usepackage[all]{xy}
\makeatletter
\numberwithin{equation}{section}
\numberwithin{figure}{section}
\theoremstyle{plain}
\newtheorem{thm}{\protect\theoremname}
  \theoremstyle{remark}
  \newtheorem{rem}[thm]{\protect\remarkname}
  \theoremstyle{plain}
  \newtheorem{lem}[thm]{\protect\lemmaname}

\date{}

\usepackage{tensor}

\newcommand{\ona}{\operatorname}

\newenvironment{conjbis}
{
  \edef\thmlabel{\thethm}
  \renewcommand{\thethm}{\thmlabel$'$}%
  \addtocounter{thm}{-1}%
   \begin{thm}}
  {\end{thm}}

\makeatother

\usepackage{babel}
  \providecommand{\lemmaname}{Lemma}
  \providecommand{\remarkname}{Remark}
\providecommand{\theoremname}{Theorem}

\begin{document}

\title{Signs of self-dual depth-zero supercuspidal representations}

\author{Manish Mishra}

\email{manish@iiserpune.ac.in}

\address{Department of Mathematics, Indian Institute of Science Education
and Research (IISER), Dr. Homi Bhabha Road, Pashan, Pune 411 008 INDIA}
\begin{abstract}
Let $G$ be a quasi-split tamely ramified connected reductive group
defined over a $p$-adic field $F$. We show that if $-1$ is in the
$F$-points of the absolute Weyl group of $G$, then self-dual supercuspidal
representations of $G(F)$ exist. Now assume further that $G$ is
unramified and that the center of $G$ is connected. Let $\pi$ be
a generic self-dual depth-zero regular supercuspidal representation
of $G(F)$. We show that the Frobenius\textendash Schur indicator
of $\pi$ is given by the sign by which a certain distinguished element
of the center of $G(F)$ of order two acts on $\pi$. 
\end{abstract}

\maketitle

\section{Introduction}

Let $\mathcal{G}$ be a group and let $(\tau,V)$ be an irreducible
representation of $\mathcal{G}$. If $\tau$ is self-dual, i.e., it
is isomorphic to its contragradient $\check{\tau}$, then there exists
a non-degenerate $\mathcal{G}$-invariant bilinear form $B:V\times V\rightarrow\mathbb{C}$
which is unique up to scalars. It is thus either symmetric or skew
symmetric. The \textit{sign} or the \textit{Frobenius\textendash Schur
indicator} $\ona{sgn}(\tau)$ of $\tau$ is defined to be $+1$ (resp.
$-1$) according as $B$ is symmetric (resp. skew-symmetric). When
$\mathcal{G}$ is a finite group, 
\[
\ona{sgn}(\tau)=\frac{1}{\mid\mathcal{G}\mid}\sum_{g\in\mathcal{G}}\omega_{\tau}(g^{2}),
\]
 where $\omega_{\tau}$ denotes the character of $\tau$. The expression
on the right hand side of the above equality is zero when $\tau$
is not self-dual. 

Now let $G$ be a connected reductive group defined over a local or
finite field $F$ and let $\pi$ be a smooth irreducible representation
of $G(F)$. When $\pi$ is self-dual and also generic, i.e., it admits
a Whittaker model, D. Prasad introduced the idea of studying the sign
by the action of certain order two element of the center $Z(F)$ of
$G(F)$ \cite{Pra98,Pra99}. See \cite[Conjecture 8.3]{GR10} for
a possible connection of this element with the Deligne-Langlands local
root number. 

Now let $G$ be unramified, $F$ be $p$-adic and let $\pi$ be an
irreducible regular depth-zero supercuspidal representation. Regular
depth-zero supercuspidal representations are the ones which arise
from certain Deligne-Lusztig cuspidal representations of finite reductive
groups. These representations were studied by DeBacker and Reeder
\cite{DeRe09}. Assuming further that $G$ has connected center and
that $\pi$ is self-dual generic, we show in Theorem \ref{thm:main}
that the sign of $\pi$ is given by the central character $\omega_{\pi}$
evaluated at a certain order two element $\epsilon$ of the center
of $G(F)$. 

The main idea in the proof of Theorem \ref{thm:main} is to reduce
the problem to a question about finite reductive groups and use Prasad's
result in that setting. In Prasad's result, we first observe that
the central element $\epsilon$ has an explicit description in terms
of the root data. The assumption of the genericity of $\pi$ is used
to ensure - by a result of DeBacker and Reeder - that the finite reductive
group in the inducing data of $\pi$ has the same root system as $G$.
We finally use Kaletha's description of regular depth-zero representations
to relate $\omega_{\pi}(\epsilon)$ to the analogous element of the
finite reductive group. 

When $G$ is quasi split and tamely ramified over $F$, we give in
Theorem \ref{thm:exist}, a sufficient condition for self-dual supercuspidal
representations of $G(F)$ to exist. We show that if $-1$ is in the
$F$-points of the absolute Weyl group of $G$, then self-dual supercuspidal
representations do exist. The proof uses Kaletha's description of
regular supercuspidal representations \cite{Kal2016} and Hakim-Murnaghan's
result about dual Yu-datum \cite{HM08}.

\section{Notations}

Let $\mathcal{G}$ be a reductive group over a local or finite field
$F$. The central character of a representation $\pi$ of $\mathcal{G}(F)$
will be denoted by $\omega_{\pi}$. The contragradient of $\pi$ will
be denoted by $\check{\pi}$. If $\pi$ is irreducible self-dual,
then its Frobenius\textendash Schur indicator will be denoted by $\ona{sgn}(\pi)$.
When $F$ is non-archimedean local, we write $\mathcal{B}(\mathcal{G},F)$
(resp. $\mathcal{B}^{\ona{red}}(\mathcal{G},F)$) to denote the Bruhat-Tits
building (resp. reduced Bruhat-Tits building) of $\mathcal{G}(F)$.
We follow the standard notations (as in \cite[Sec. 2]{Kal2016} for
instance) for parahoric subgroups of $\mathcal{G}(F)$ and their Moy-Prasad
filtrations.

\section{\label{sec:finite}finite reductive group }

Let $\mathtt{G}$ be a connected reductive group defined over a finite
field $\mathbb{F}_{q}$. We assume that center $\mathtt{Z}$ of $\mathtt{G}$
is connected. Let $\mathtt{B=TU}$ be an $\mathbb{F}_{q}$-Borel subgroup
of $\mathtt{G}$, where $\mathtt{U}$ is the unipotent radical of
$\mathtt{B}$ and $\mathtt{T}$ is an $\mathbb{F}_{q}$-maximal torus
of $\mathtt{G}$ contained in $\mathtt{B}$. We denote the adjoint
torus by $\mathtt{T}_{\ona{ad}}$. The character lattice of $\mathtt{T}$
(resp. $\mathtt{T}_{\ona{ad}}$) will be denoted by $X^{*}(\mathtt{T})$
(resp. $X^{*}(\mathtt{T}_{\ona{ad}})$). 
\begin{thm}
[Prasad]There exists an element $s_{0}$ in $\mathtt{T}(\mathbb{F}_{q})$
such that it operates by $-1$ on all the simple root spaces of $\mathtt{U}$.
Further, $t_{0}:=s_{0}^{2}$ belongs to $\mathtt{Z}(\mathbb{F}_{q})$
and $t_{0}$ acts on an irreducible, generic, self-dual representation
by $1$ iff the representation is orthogonal. 
\end{thm}
From the short exact sequence
\[
\xymatrix{1\ar@{->}[r] & \mathtt{Z}\ar@{->}[r] & \mathtt{T}\ar@{->}[r] & \mathtt{T}_{\ona{ad}}\ar@{->}[r] & 1}
,
\]
we get the long exact sequence
\[
\xymatrix{1\ar@{->}[r] & \mathtt{Z}(\mathbb{F}_{q})\ar@{->}[r] & \mathtt{T}(\mathbb{F}_{q})\ar@{->}[r] & \mathtt{T}_{\ona{ad}}(\mathbb{F}_{q})\ar@{->}[r] & \ona{H}^{1}(\Gamma,\mathtt{Z})\ar@{->}[r] & \cdots}
.
\]
 Since $\mathtt{Z}$ is connected, $\ona{H}^{1}(\Gamma,\mathtt{Z})$
is trivial by Lang's theorem. Therefore,
\[
\xymatrix{1\ar@{->}[r] & \mathtt{Z}(\mathbb{F}_{q})\ar@{->}[r] & \mathtt{T}(\mathbb{F}_{q})\ar@{->}[r] & \mathtt{T}_{\ona{ad}}(\mathbb{F}_{q})\ar@{->}[r] & 1}
\]
 is exact. Let $\check{\rho}$ denote half the sum of positive co-roots.
Let $s^{\prime}$ be the element of $\mathtt{T}_{\ona{ad}}(\mathbb{F}_{q})=\ona{Hom}(X^{*}(\mathtt{T}_{\ona{ad}}),\mathbb{G}_{m})(\mathbb{F}_{q})$
given by 
\[
\chi\in X^{*}(\mathtt{T}_{\ona{ad}})\mapsto(-1)^{\langle\chi,\check{\rho}\rangle}\in\mathbb{G}_{m}.
\]
 Let $s$ denote any pull back of $s^{\prime}$ in $\mathtt{T}(\mathbb{F}_{q})$.
Then $s$ operates by $-1$ on all simple root spaces of $\mathtt{U}$.
The element $t:=s^{2}\in\mathtt{Z}(\mathbb{F}_{q})$ has the description
\[
\chi\in X^{*}(\mathtt{T})\mapsto(-1)^{\langle\chi,2\check{\rho}\rangle}\in\mathbb{G}_{m}.
\]

We can thus rewrite the above Theorem as

\begin{conjbis}

\label{Thm:Prasad}Let $\pi$ be an irreducible generic representation
of $\mathtt{G}(\mathbb{F}_{q})$. Then $\ona{sgn}(\pi)=\omega_{\pi}(t)$. 

\end{conjbis}
\begin{rem}
The assumption in Theorem \ref{Thm:Prasad} that $\mathtt{Z}$ is
connected cannot be entirely dropped, as shown in the counter example
in \cite[Sec. 9]{Pra98}.
\end{rem}

\section{\label{sec:Regular}Regular depth-zero representations}

Let $G$ be an unramified connected reductive group defined over a
$p$-adic field $F$. Assume that the center $Z$ of $G$ is connected.
Let $\mathfrak{f}$ denote the residue field of $F$.

\subsection{\label{sub:Construction}Construction of regular depth-zero supercuspidal }

For the definition of \textit{regular} depth-zero supercuspidal representations
and the details of the construction in this section, see \cite[Sec. 3.2.3]{Kal2016}.
Let $S$ be an elliptic maximal torus of $G$ and let $\theta:S(F)\rightarrow\mathbb{C}^{\times}$
be a depth-zero character. Let $S(F)_{0}$ be the Iwahori subgroup
of $S(F)$. Assume that $\theta$ is \textit{regular}, i.e., the stabilizer
of $\theta\mid_{S(F)_{0}}$ in $N(S(F),G(F))/S(F)$ is trivial, where
$N(S(F),G(F))$ denotes the normalizer of $S(F)$ in $G(F)$. The
restriction of $\theta\mid_{S(F)_{0}}$ factors through a character
$\bar{\theta}$ of $S(F)_{0:0+}$. Let $x\in\mathcal{B}^{\ona{red}}(G,F)$
be the vertex associated to $S$. The group $G(F)_{0:0+}$ is the
$\mathfrak{f}$-points of a connected reductive $\mathfrak{f}$-group
$\mathtt{G}_{x}$ and $S(F)_{0:0+}$ is the $\mathfrak{f}$-points
of an elliptic maximal $\mathfrak{f}$-torus $\mathtt{S}^{\prime}$
of $\mathtt{G}_{x}$. Let $\kappa(S,\theta)$ denote the irreducible
cuspidal Deligne-Lusztig representation of $\mathtt{G}_{x}(\mathfrak{f})$
associated to the pair $(\mathtt{S}^{\prime},\bar{\theta})$. Denote
again by $\kappa(S,\theta)$ its inflation to $G(F)_{x,0}$. This
representation extends to a representation $\tilde{\kappa}(S,\theta)$
of $Z(F)G(F)_{x,0}=G(F)_{x}$. 
\begin{lem}
\cite[Lemma 3.17, 3.22]{Kal2016} \label{lem:kal}The representation
$\pi(S,\theta):=\ona{c-Ind}_{G(F)_{x}}^{G(F)}\tilde{\kappa}(S,\theta)$
is irreducible (and hence supercuspidal) and every regular depth-zero
supercuspidal representation is of this form. 
\end{lem}

\subsection{\label{sub:Sign}Sign}

Choose a system of positive roots $\Phi^{+}(G,S)$ for the set of
roots $\Phi(G,S)$. Let $\check{\rho}$ denote half the sum of positive
roots and let $t$ denote the element $2\check{\rho}(-1)\in$ $Z(F)$
(see \cite[Sec. 8.5]{GR10}). 
\begin{thm}
\label{thm:main}Let $\pi$ be a generic regular depth-zero self-dual
supercuspidal representation of $G(F)$. Then $\ona{sgn}(\pi)=\omega_{\pi}(t)$.\end{thm}
\begin{proof}
By Lemma \ref{lem:kal}, the representation $\pi$ arises out of a
pair $(S,\theta)$ as in Section \ref{sub:Construction}. Let $x\in\mathcal{B}^{\ona{red}}(G,F)$
be the vertex associated to $S$. The normalizer in $G(F)$ of $G(F)_{x,0}$
is $G(F)_{x}$ \cite[Lemma 3.3]{Yu01}. Therefore $G(F)_{x}$ is self
normalizer. Since $\check{\pi}=\ona{c-Ind}_{G(F)_{x}}^{G(F)}\check{\tilde{\kappa}}$,
$\pi\cong\check{\pi}$ iff there exists $g\in G(F)$ such that $(G(F)_{x},\tilde{\kappa})$
is conjugate by an element $g\in G(F)$ to the pair $(G(F)_{x},\check{\tilde{\kappa}})$
(by \cite[Theorem 6.7 ]{HM08} for instance without any hypothesis).
But then $g\in G(F)_{x}$ since $G(F)$ is self normalizer. Therefore,
$\pi(S,\theta)$ is self-dual iff $\tilde{\kappa}(S,\theta)$ is so.
Thus, $\ona{sgn}(\pi)=\ona{sgn}(\tilde{\kappa})$. Since $\pi(S,\theta)$
is generic, the vertex $x$ associated to $S$ is hyperspecial \cite[Theorem 1.1]{DeRe10}
(also \cite[Lemma 6.1.2]{DeRe09}). Therefore the root system $\Phi(G,S)$
can be identified with $\Phi(\mathtt{G}_{x},\mathtt{S}^{\prime})$.
Let $\Phi^{+}(\mathtt{G}_{x},\mathtt{S}^{\prime})$ be the positive
roots of $\Phi(\mathtt{G}_{x},\mathtt{S}^{\prime})$ under the identification.
Let $\underline{\check{\rho}}$ be half the sum of positive roots
of $\Phi(\mathtt{G}_{x},\mathtt{S}^{\prime})$ and $\underline{t}=2\underline{\check{\rho}}(-1)\in\mathtt{Z}_{x}(\mathbb{F}_{q})$.
Since $x$ is hyperspecial, $\mathtt{Z}$ is connected implies $\mathtt{Z}_{x}$
is connected. Also, $\kappa(S,\theta)$ is generic \cite[Lemma 6.1.2]{DeRe09}.
We therefore have by Theorem \ref{Thm:Prasad} that $\ona{sgn}(\tilde{\kappa})=\omega_{\tilde{\kappa}}(\underline{t})$.
But $\omega_{\tilde{\kappa}}(\underline{t})=\bar{\theta}(\underline{t})=\theta(t)$.
The Theorem now follows because $\omega_{\pi}=\theta\mid_{Z(F)}$
by \cite[Fact 3.38]{Kal2016}.
\end{proof}

\section{\label{sec:Existence}Existence of self-dual representations}

Let $G$ be a quasi-split tamely ramified connected reductive group
over a $p$-adic field $F$. Let $\Omega(S,G)$ be the absolute weyl
group. In \cite[Sec. 3.2.1]{kaletha}, a vertex $x\in\mathcal{B}^{\ona{red}}(G,F)$
is called \textit{superspecial} if it is special in $\mathcal{B}^{\ona{red}}(G,E)$,
where $E$ is any finite Galois extension of $F$ splitting $G$. 
\begin{thm}
\label{thm:exist}If $-1\in\Omega(S,G)(F)$, then self-dual supercuspidal
representations of $G(F)$ exist. \end{thm}
\begin{proof}
Let $(S,\theta)$ be a tame regular elliptic pair \cite[Def. 3.23]{Kal2016}
such that $S$ is \textit{relatively unramified} \cite[Sec. 3.2.1]{Kal2016}
and the point $x\in\mathcal{B}^{\ona{red}}(G,F)$ associated to $S$
is superspecial. Let $\pi(S,\theta)$ be the associated regular supercuspidal
representation. By \cite[Theorem 4.25]{HM08}, $\check{\pi}(S,\theta)\cong\pi(S,\theta^{-1})$.
By \cite[Lemma 3.37]{Kal2016}, $\pi(S,\theta)\cong\pi(S,\theta^{-1})$
iff $(S,\theta)$ is $G(F)$-conjugate to $(S,\theta^{-1})$. By \cite[Lemma 3.11]{Kal2016}
$\Omega(S,G)(F)\cong N(S,G)(F)/S(F)$, where $N(S,G)$ denotes the
normalizer of $S$ in $G$. Thus if $-1\in\Omega(S,G)(F)$, then it
follows that $(S,\theta)$ is $G(F)$-conjugate to $(S,\theta^{-1})$. \end{proof}
\begin{rem}
If the root system of $G$ is of type $B_{n}$, $C_{n}$, $E_{7}$,
$E_{8}$, $G_{2}$ or $D_{n}$ ($n$-even), then the longest weyl
group element of $G$ is $-1$. 
\end{rem}

\begin{rem}
When $G=\ona{GL}_{n}$, Adler \cite{Adler97} showed that the necessary
and sufficient condition for self-dual regular supercuspidal representations
of $G(F)$ to exist is that either $n$ or the residue characteristic
of $F$ be even. 
\end{rem}

\section{Acknowledgment}

The author is very thankful to Sandeep Varma, Dipendra Prasad and
Steven Spallone for many helpful conversations. 

\bibliographystyle{plain}
\bibliography{summary}

\end{document}